\documentclass{amsart}

\usepackage{amsmath,amssymb,amsthm,graphicx}

\newtheorem{theorem}{Theorem}
\newtheorem{lemma}[theorem]{Lemma}

\newtheorem{dus}[theorem]{Corollary}

\theoremstyle{definition}
\newtheorem{definition}[theorem]{Definition}
\newtheorem{example}[theorem]{Example}

\theoremstyle{remark}

\numberwithin{equation}{section}

\begin{document}

\title{Local reflexion spaces}

\author{Jan Gregorovi\v c}

\keywords{}

\dedicatory{}

\begin{abstract}
A reflexion space is generalization of a symmetric space introduced by O. Loos in \cite{odk1}. We generalize locally symmetric spaces to local reflexion spaces in the similar way. We investigate, when local reflexion spaces are equivalently given by a locally flat Cartan connection of certain type.
\end{abstract}

\maketitle

There are several equivalent definitions of symmetric spaces and locally symmetric spaces. For example, an (affine) locally symmetric space is a connected smooth manifold $M$ with a torsion-free linear connection with parallel curvature. Another definition is, that a (homogeneous) locally symmetric spaces is a locally flat Cartan geometry of type $(G,H)$ on a connected manifold $M$ if there is $h\in H$ such, that $h^2=id_G$ and $H$ is open in the centralizer of $h$ in $G$. The equivalence of these two definitions can be found for example in \cite{odk6}.

The reflexion spaces were introduced by O. Loos in \cite{odk1}. He found, that reflexion spaces are equivalent to fibre bundles associated to homogeneous symmetric space $G\to G/H$. If the Lie group $G$ acts transitively on the reflexion space or equivalently $H$ acts transitively on the fiber, then if we denote $K$ stabilizer of one point of the reflexion space, the structure of reflexion space is equivalently given by a Maurer-Cartan form of $G\to G/K$ i.e. by a flat Cartan connection of type $(G,K)$.

Now, we introduce a local version of the reflexion spaces and investigate, under which conditions they are equivalently given by a locally flat Cartan connection of certain type.

\begin{definition}
Let $M$ be a connected smooth manifold, $N$ a neighborhood of the diagonal in $M\times M$ and $S: N \to M$ a smooth mapping. We denote 
\[ S(x,y)=S_xy=S^yx\]
and we say that $S_x$ is a (local) reflexion at $x$. We call $(M,S)$ a local reflexion space under the following three conditions:
\begin{itemize}
\item[(A1)] $S_xx=x$

\item[(A2)] If $U_x:=\{y: (x,y)\in N\}$, then $S_x$ is a diffeomorphism of $U_x$ satisfying $S_x(S_xy)=y$ for all $y\in U_x$.

\item[(A3)] There is a neighborhood $W$ of the diagonal in $M\times M\times M$ such, that 
\[ S_xS(y,z)=S(S_xy,S_xz)\]
holds for all $(x,y,z)\in W$.
\end{itemize}

Let $(M,S)$ and $(M',S')$ be two local reflexion spaces and $U\subset M$. Then $f: U \to M'$ is a local homomorphism of local reflexion spaces (we will say only homomorphism), if $f((U\times U) \cap N)\subset N'$ and
\[ f(S_xy)=S'_{f(x)}f(y)\]
for $(x,y)\in (U\times U) \cap N$.
\end{definition}

The meaning of conditions (A2) and (A3) is, that all (local) reflexions have to be involutive local automorphisms of local reflexion spaces.

There are the following examples of local reflexion spaces:

\begin{example}
Let $(p: \mathcal{G}\to M, \omega)$ be a locally flat Cartan geometry of type $(G,K)$ and assume, that there is $h\in K$ satisfying $h^2=id_G$ and $hk=kh$ for any $k\in K$. 

Since the Cartan geometry is locally flat, there is an atlas of $M$ such, that the images of charts are open subsets of $G/K$ and transition maps are restrictions of left actions of elements of $G$. In particular for all $x\in M$, there are local coordinates $V_x \subset G/K$ of some neighborhood of $x$ centered at $x$.

We denote $\bar{V}_x\subset \mathfrak{g}$ some neighborhood of $0$ such, that of $p(\exp(\bar{V}_x))\subset V_x$ and let $\bar{U}_x \subset \bar{V}_x$ be such, that 
\[ \exp(X)\exp(Ad(h)(-X))\exp(Y)K\in V_x\]
for all $X,Y\in \bar{U}_x$. Then we define 
\[ N:=\bigcup_{x\in M}(p(\exp(\bar{U}_x)),p(\exp(\bar{U}_x))),\]
so $N$ is a neighborhood of diagonal in $M\times M$ and we define 
\[S_{\exp(X)K}{\exp(Y)K}:=\exp(X)\exp(Ad(h)(-X))\exp(Y)K.\]
Since $K$ commutes with $h$, the definition is correct.

We show, that $(M,S)$ is a local reflexion space. Let $tfK, tgK\in p(\exp(\bar{U}_x))$ and $fK, gK \in p(\exp(\bar{U}_y))$ be two different coordinates of the same points of $M$, where the transition map between those coordinates is a left action of $t\in G$, then 
\[ S_{tfK}tgK=tfhf^{-1}t^{-1}tgK=tS_{fK}gK\]
i.e. the definition of $S$ does not depend on the choice of coordinates.

Let $\bar{W}_x \subset \bar{U}_x$ be such, that 
\[ \exp(X)\exp(Ad(h)(-X))\exp(Ad(h)Y)\exp(-Y)\exp(Z)K\in V_x\]
for all $X,Y,Z\in \bar{W}_x$.  We define
\[ W=\bigcup_{x\in M} (p(\exp(\bar{W}_x)),p(\exp(\bar{W}_x)),p(\exp(\bar{W}_x))).\]
Checking that (A1), (A2) and (A3) holds, is then an easy computation. 

For later use, we will notice that we can reconstruct the local Cartan geometry, under certain conditions. Consider the one parameter subgroup $f_t=\exp(tX)$. Then
\[ \frac{d}{dt}|_{t=0}S_{f_tK}S_{eK}gK=R_X(gK)-R_{Ad(h)X}(gK),\]
where $R_X$ is the projection of right invariant vector field of $X\in \bar{W}_x$ on $p(\exp(\bar{W}_x))$. Since $h^2=id_G$, we denote $\mathfrak{g}^-$ the $-1$ eigenspace of $Ad(h)$. Then for $X\in \mathfrak{g}^-$ is $\frac{d}{dt}|_{t=0}S_{f_tK}S_{eK}gK=2R_X(gK)$ i.e. 
\[ S_{\exp(X)K}S_{eK}gK=\exp(2X)gK.\]
Thus if $\mathfrak{g}^-$ generates the Lie algebra $\mathfrak{g}$ by the Lie bracket, we can reconstruct the right invariant vector fields from $S_xS_e$ action i.e. we can reconstruct locally flat Cartan geometry of type $(\mathfrak{g},\mathfrak{k})$.
\end{example}

We choose the following representative for the equivalence class of the Cartan geometries obtained in the example:

\begin{definition}
We say that a local reflexion space $(M,S)$ is locally homogeneous, if it is locally equivalent (as in previous example) to a locally flat Cartan geometry $(p: \mathcal{G}\to M, \omega)$ of type $(G,K)$ such, that

\begin{itemize}
\item[(H1)] there is $h\in K$ such, that $h^2=id_G$, $hk=kh$ for any $k\in K$

\item[(H2)]  the  $-1$ eigenspace of $Ad(h)$ in $\mathfrak{g}$ generates whole $\mathfrak{g}$ by the Lie bracket

\item[(H3)] $G/K$ is connected, simply connected and the maximal normal subgroup of $G$ contained in $K$ is trivial.
\end{itemize}
\end{definition}

We are interested, when are the local reflexion spaces locally homogeneous? The answer is the following:

\begin{theorem}
Let $(M,S)$ be a local reflexion space and let $\mathfrak{g}_x$ be a Lie subalgebra of Lie algebra of vector fields on some neighborhood of $x\in M$ generated by $\frac{d}{dt}|_{t=0}S_{x(t)}S_x,$ where $x(t)$ is a smooth curve such, that $x(0)=x$. If for any $x\in M$ is $\mathfrak{g}_x(x)=T_xM$, then $(M,S)$ is a locally homogeneous local reflexion space.
\end{theorem}

Before we start the proof, we fix the following notation:

\begin{itemize}
\item choose $W$ as in condition (A3) in definition, and denote $W_x$ a neighborhood of $x$ such, that $W_x\times W_x\times W_x\subset W$

\item $V_x:=\{S_yz: y,z\in W_x\}$

\item we denote by $X,Y,\dots$ vector fields on $U\subset M$ and we assume, that we have chosen for any point $x,y\dots \in U$ a smooth curve $x(t),y(t),\dots$ in $U$ satisfying $x(0)=x,y(0)=y,\dots$ and $x'(0)=X(x),y'(0)=Y(y),\dots$

\item we shall write $TS_xY:=\frac{d}{dt}|_{t=0}S_xy(t)$, $TS^xY:=\frac{d}{dt}|_{t=0}S_{y(t)}x$

\item we denote $XY$ the differential operator acting on $f: U\subset M\to \mathbb{R}$ as $X(Yf)$

\item we denote $T^2S(X,Y)$ the differential operator acting on $f: U\subset M\to \mathbb{R}$ as 
\[ 
\begin{split}
(T^2S(X,Y))f(S_xy)&=\frac{d}{dt}|_{t=0}\frac{d}{ds}|_{s=0}f(S_{x(t)}y(s))\\
&=(TS_xY)(TS^yX)f(y)=(TS^yX)(TS_xY)f(x)
\end{split}
\]

\item we denote $R_x(X)$ a vector field extension of $X\in T_xM$ given by 
\[ R_x(X)(y):=\frac12 TS^{S_xy}X.\]
\end{itemize}

We see, that the axioms (A1), (A2) and (A3) are defined for all points of $W_x$ and further we shall restrict ourselves to $W_x$ if not stated otherwise.

We call a map $\phi$ defined on an interval $(a,b)\subset \mathbb{R}$ containing zero with values in the pseudogroup of locally defined diffeomorphisms of $M$ a local one parameter subgroup of local automorphisms on $W_x$, if $\phi$ satisfies:
\[ \phi_0=id_{V_x},\ \phi_{t+s}=\phi_t\circ\phi_s,\  \phi_t(S_pq)=S_{\phi_t(p)}\phi_t(q)\]
for all $p,q\in W_x$. Then we obtain an infinitesimal version of local automorphisms of local reflexion spaces by differentiation of $\phi$:

\begin{definition}
Let $(M,S)$ be a local reflexion space. We say that a vector field $X$ defined on $V_x$ is an infinitesimal automorphism if 
\[ X(S_pq)=TS_pX(q)+TS^qX(p)\]
for all $p,q\in W_x$.
\end{definition}

The following lemma shows equivalence between the local one parameter subgroups of local automorphisms and the infinitesimal automorphisms. Moreover, we obtain condition, when they are generated by reflexions:

\begin{lemma}\label{odkeqv}
Let $\phi_t$ be a local one parameter subgroup of locally defined diffeomorphisms given as a flow of some vector field $X$ on $V_x$. Then $\phi_t$ is a local one parameter subgroup of local automorphisms at $W_x$ if and only if $X$ is an infinitesimal automorphism. If $X$ is an infinitesimal automorphism and $(S_x)^*X=-X$, then 
\[S_{\phi_t(x)}S_x=\phi_{2t}.\]
\end{lemma}
\begin{proof}
One of the implications is obvious, we prove the other one. Let 
\[ \gamma(t):=\phi_{-t}(S_{\phi_t(p)}\phi_t(q)).\]
Then 
\[ 
\begin{split}
\gamma'(t)&=-X(\gamma(t))+T\phi_{-t}(TS_pX(q)+TS^qX(p))\\
&=-X(\gamma(t))+T\phi_{-t}X(S_{\phi_t(p)}\phi_t(q))\\
&=-X(\gamma(t))+T\phi_{-t}\circ X\circ \phi_t(\gamma(t))=0.
\end{split}
\]
Thus the curve $\gamma$ is constant and 
\[ \phi_{-t}(S_{\phi_t(p)}\phi_t(q))=\gamma(0)=S_pq.\]
Then for the flow $Fl^X$ of $X$ holds
\[ 
S_{Fl_{t}^{X}(x)}S_x=Fl_{t}^{X}S_xFl_{-t}^{X}S_x=Fl_{t}^{X}Fl_{-t}^{(S_x)^*(X)}
\]
If $(S_x)^*X=-X$, then
\[ 
S_{Fl_{t}^{X}(x)}S_x=Fl_{t}^{X}Fl_{-t}^{-X}=Fl_{2t}^{X}=\phi_{2t}.
\]
\end{proof}

We see, that $R_x(X)(y)$ is a candidate for an infinitesimal automorphism. We show that this is indeed the case:

\begin{lemma}\label{fin3l}
\begin{enumerate}
\item The set $\mathcal{D}_x$ of all infinitesimal automorphisms on $V_x$ is a Lie subalgebra of the Lie algebra of vector fields on $V_x$.

\item $(S_x)^*$ is an involutive automorphism of $\mathcal{D}_x$ and we denote $\mathfrak{g}_x^-$ the $-1$ eigenspace of $(S_x)^*$.

\item Let $T_x^-M+T_x^-M$ be the decomposition of $T_xM$ with respect to the $-1$ and $1$ eigenspaces of $(S_x)^*$. Then $TM=T^-M+T^+M$ is a decomposition to subbundles, which is preserved by the local reflexions.

\item $R_x$ is an isomorphism of the vector spaces $T_x^-M$ and $\mathfrak{g}_x^-$ and for $X\in T_x^+M$ is $R_x(X)=0$. 

\item $[[\mathfrak{g}_x^-,\mathfrak{g}_x^-],\mathfrak{g}_x^-]\subset \mathfrak{g}_x^-$ and, moreover, the Lie subalgebra $\mathfrak{g}_x\subset D_x$ generated by $\mathfrak{g}_x^-$ is a finite dimension Lie subalgebra of $D_x$. In particular, $\mathfrak{g}_x=\mathfrak{g}_x^-+[\mathfrak{g}_x^-,\mathfrak{g}_x^-]$ and any ideal of $\mathfrak{g}_x$ contained in $[\mathfrak{g}_x^-,\mathfrak{g}_x^-]$ is contained in the center of $\mathfrak{g}_x$.

\item Let $\phi$ be a local automorphism given by a composition of local reflexions such, that $\phi(x)=z$. Then $T\phi: \mathfrak{g}_x \to \mathfrak{g}_z$ is an isomorphism of Lie algebras.
\end{enumerate}
\end{lemma}
\begin{proof}
(1) For $Y\in \mathcal{D}_x$, 
\[ \begin{split}
(TS^qP)(Y)(S_pq)&=(TS^qP)(TS_pY)(q)+(TS^qP)(TS^qY)(p)\\
&=T^2S(P,Y)(S_pq)+TS^q(PY)(p)
\end{split}\] 
and in the same way obtain
\[ (TS_pQ)Y(S_pq)=T^2S(Y,Q)(S_pq)+TS_p(QY)(q).\]
For $X,Y\in \mathcal{D}_x$,
\[ 
\begin{split}
[X,Y](S_pq)&=XY(S_pq)-YX(S_pq)\\
&=(TS_pX(q)+TS^qX(p))Y(S_pq)-(TS_pY(q)+TS^qY(p))X(S_pq)\\
&=TS^q(XY)(p)+T^2S(X,Y)(S_pq)+TS_p(XY)(q)+T^2S(Y,X)(S_pq)\\
&-TS^q(YX)(p)-T^2S(Y,X)(S_pq)-TS_p(YX)(q)-T^2S(X,Y)(S_pq)\\
&=TS^q[X,Y](p)+TS_p[X,Y](q)
\end{split}
\]
i.e. we have shown that $[X,Y]\in \mathcal{D}_x$.

(2) Differentiating $S_{x}S_yz(t)=S_{S_{x}y}S_{x}z(t)$ we obtain 
\[ TS_xTS_yZ=TS_{S_xy}TS_xZ\]
and differentiating $S_{x}S_{y(t)}z=S_{S_{x}y(t)}S_{x}z$ we obtain
\[TS_xTS^zY=TS^{S_xz}TS_xY\]
Then for $X\in \mathcal{D}_x$ 
\[ 
\begin{split}
(TS_x\circ X(S_pq)\circ S_x)&=TS_xX(S_xS_pq)=TS_xX(S_{S_xp}{S_xq})\\
&=TS_x TS^{S_xq}X(S_xp)+TS_xTS_{S_xp}X(S_xq)\\
&=TS^q TS_xX(S_xp)+TS_pTS_xX(S_xq)\\
&=TS_q(TS_x\circ X\circ S_x)(p)+TS_p(TS_x\circ X\circ S_x)(q),
\end{split}\]
i.e. we have shown that $(S_x)^*$ is an automorphism of $\mathcal{D}_x$. Differentiating $S_{x}S_{x}y(t)=y(t)$ we obtain
\[(TS_x)^2Y=Y.\]
Thus $(S_x)^*|_{T_xM}=TS_x$ has only eigenvalues $\pm 1$ and, since $S_x^2=id$, $((S_x)^*)^2=id$.

(3) Differentiating $S_{x(t)}x(t)=x(t)$ we obtain
\[ TS_xX+TS^xX=X(x).\]
Thus $TS^x$ is a projection from $T_xM\to T_x^-M$ with kernel $T_x^+M$ and $T^-M+T^+M$ is a decomposition of $TM$ to subbundles. Further, we have shown that $TS_y (TS_xX)=TS_{S_yx}(TS_yX)$, so we see, that the reflexions preserve the decomposition $TM=T^-M+T^+M$.

(4) Differentiating $S_{S_{x(t)}y}z=S_{x(t)}S_yS_{x(t)}z$ we obtain 
\[ TS^zTS^yX=TS^{S_yS_xz}X+TS_xTS_yTS^zX\]
and differentiating $S_{x(t)}S_{x(t)}y=y$ we obtain
\[ TS^{S_xy}X+TS_xTS^yX=0\]
So 
\[ 
\begin{split}
R_x(TS^x(X))(y)&=\frac14 TS^{S_xy}TS^xX\\
&=\frac14 TS^{S_xS_xS_xy}X+\frac14 TS_xTS_xTS^yX=R_x(X)(y),
\end{split}\]
Thus for $X\in T^+M$ we obtain $R_x(X)=0$.

Next we show $R_x(X)\in \mathfrak{g}_x^-$. Differentiating $S_{x(t)}S_yz=S_{S_{x(t)}y}S_{x(t)}z$ we obtain 
\[TS^{S_yz}X=TS^{S_xz}TS^yX+TS_{S_xy}TS^zX.\]
Thus 
\[ 
\begin{split}
2R_x(X)(S_yz)&=TS^{S_xS_yz}X=TS^{S_{S_xy}S_xz}X\\
&=TS^{S_xS_xz}TS^{S_xy}X+TS_{S_xS_xy}TS^{S_xz}X\\
&=2(TS^zR_x(X)(y)+TS_yR_x(X)(z)),
\end{split}\]
and
\[ 2(TS_x\circ R_x(X)\circ S_x)(y)=TS_xTS^{S_xS_xy}X=-TS^{S_xy}X=-2R_x(X)(y).\]
Since $R_x(X)(x)=TS^x(X)$, the map $R_x$ is injective. 

Differentiating $S_{x}S_{y(t)}z=S_{S_{x}y(t)}S_{x}z$ we obtain 
\[TS_xTS^zY=TS^{S_xz}TS_xY.\]
Then for $X\in \mathfrak{g}_x^-$, we may conclude:
\[
\begin{split}
-X(y)&=TS_x\circ X(y) \circ S_x=TS_xX(S_xy)\\
&=TS_xTS^yX(x)+TS_xTS_xX(y)\\
&=TS^{S_xy}TS_xX(x)+X(y)=-TS^{S_xy}X(x)+X(y).
\end{split}\]
Thus $X(y)=R_x(X(x))(y)$ and $R_x$ is surjective.

(5) Since
\[
\begin{split}
(S_x)^*([[R_x(X),R_x(Y)],R_x(Z)])&=[[(S_x)^*(R_x(X)),(S_x)^*(R_x(Y))],(S_x)^*(R_x(Z))]\\
&=-[[R_x(X),R_x(Y)],R_x(Z)],
\end{split}
\]
we get $[[\mathfrak{g}_x^-,\mathfrak{g}_x^-],\mathfrak{g}_x^-]\in \mathfrak{g}_x^-$. Further, $[[R_x(X),R_x(Y)],R_x(Z)]$ is linear in all entries, thus the Lie algebra $\mathfrak{g}_x$ generated by $\mathfrak{g}_x^-$ is finite dimensional and $\mathfrak{g}_x=\mathfrak{g}_x^-+[\mathfrak{g}_x^-,\mathfrak{g}_x^-]$. From the isomorphism $\mathfrak{g}_x^-=T_x^-M$ we get $[\mathfrak{g}_x^-,\mathfrak{g}_x^-]\subset End(T_x^-M)$ and any ideal of $\mathfrak{g}_x$ contained in $[\mathfrak{g}_x^-,\mathfrak{g}_x^-]$ is contained in center of $\mathfrak{g}_x$.

(6) $T\phi$ induces a vector space isomorphism between $T_x^-M$ and $T_z^-M$. Since 
\[ 
\begin{split}
T\phi X(S_pq)&=T\phi(TS^{q}X(p)+TS_{p}X(q))\\
&=TS^{\phi(q)} T\phi(X)(\phi(p))+TS_{\phi(p)}T\phi(X)(\phi(q)),
\end{split}
\]
it maps $\mathfrak{g}^-_x$ to $\mathfrak{g}^-_{z}$. Further,
\[ 2(TS_w\circ R_x(X)\circ S_w)(y)=TS_wTS^{S_xS_wy}X=TS^{S_{S_wx}y}T_xS_w(X)=2R_{S_wx}(T_xS_w(X))(y).\]
Thus, if $\phi$ is a composition of local reflexions, then it is compatible with the Lie bracket of vector fields and the claim follows. 
\end{proof}

Thus if $\mathfrak{g}_x(x)=T_xM$ for any $x\in M$, the previous two lemmas show, that there are infinitesimal automorphisms in all directions. Now we need a version of the Lie second fundamental theorem for this situation:

\begin{lemma}\label{fin1}
Let $\mathfrak{g}$ be a finite dimensional Lie subalgebra of Lie algebra $\mathcal{D}_x$. Then there is a connected, simply connected Lie group $G$ with the Lie algebra $\mathfrak{g}$, an open subset $U\in G$ and a local left action $l: U\times W_x \to M$.
\end{lemma}
\begin{proof}
The lemma is a local version of \cite{odk5}[Lemma 2.3]. For the convenience of the reader we include the proof below. The details on the parallel transport can be found in \cite{odk4}[Chapter 9].

Let $G$ be a connected, simply connected Lie group with Lie algebra $\mathfrak{g}$. There is the integrable distribution $(L_X,X)$ on $G\times V_x$, where $L_X$ is a left invariant vector field corresponding to $X\in \mathfrak{g}$. We will denote $L(y)$ the leaf through $(e,y)$. The $pr_1: G\times V_x\to G$ is a trivial fibre bundle with a flat connection (for a horizontal distribution given by $(L_X,X)$). Further, $pr_1|_{L(y)}$ is a local diffeomorphism onto an open neighborhood $Q(y)$ of $e$ in $G$. 

We will use the parallel transport $Pt(c,(g,y),t)$ with respect to the flat connection. For a curve $c:(a,b)\to G, c(0)=g$, $Pt(c,(g,y),t)$ is defined on some neighborhood $V$ of ${g}\times V_x\times {0}$ in ${g}\times V_x\times \mathbb{R}$.

Let $c: [0,1]\to Q(y)$ be a piecewise smooth curve with $c(0)=e$. Since ${e}\times{y}\times [0,1]\subset V$, then $Pt(c,(e,.),0)$ is defined for points in a open subset $U(y)$ containing $y$ and is a diffeomorphism of ${c(0)}\times U(y)$ onto its image ${c(1)}\times U'$. We choose $U(y)$ maximal with this property. Since the connection is flat, the parallel transport depends on the homotopy classes of the curve $c$ (with fixed end points). Thus, $Pt(c,(e,.),0)$ defines the map $\gamma_y(c):=pr_2\circ Pt(c,(e,.),0): U(y)\to U'$.

Now let $\bar{V}$ be a neighborhood of $0$ in $\mathfrak{g}$ such, that $Fl_1^X(y)$ is defined for all $y\in W_x$ for $X\in \bar{V}$. Then there is $\bar{U}\subset \bar{V}$ such, that $U=\exp(\bar{U})\subset Q(y)$ for all $y\in W_x$. Thus $Pt(c,(c(0),.),0)$ is defined for all $y\in W_x$ and for all $c: [0,1]\to U$.

We define the local left action $l: U\times W_x \to M$ as $l(g,y)=\gamma_x(c)(y)$, where $c$ is a piecewise smooth curve with $c(0)=e$ and $c(1)=g$. Obviously, the definition is correct and it is a left action. Indeed $l(\exp(tX),y)=Fl_t^X(y)$ is the local one parameter group of local automorphisms generated by $X\in \mathfrak{g}$.
\end{proof}

As a corollary of the lemmas \ref{odkeqv}, \ref{fin3l} and \ref{fin1} we get the following:

\begin{dus}\label{ltran}
If there is $x\in M$ such that $\mathfrak{g}_x(x)=T_xM$, then the pseudogroup of locally defined diffeomorphisms generated by pairs of local reflexions acts transitively on $M$ and locally is generated by $\mathfrak{g}_x$.
\end{dus}

Now we can prove the main theorem:

\begin{proof}
Lemma \ref{odkeqv} and corollary \ref{ltran} imply, that $\mathfrak{g}_x$ are isomorphic Lie algebras for all $x\in M$, and there are local actions of $G$ from lemma \ref{fin1} around all points. We denote $K$ the connected component of identity of stabilizer of some point $x\in M$. We have shown that maximal normal subgroup of $G$ contained in $K$ is contained in center of $G$ and we factor out this part to satisfy condition (H3). The local actions of $G$ provide an atlas of $M$ such that the images of charts are open subsets of $G/K$ and transition functions are elements of $G$. If we glue the pullbacks of restrictions of the images of the charts in $G\to G/K$ using the same transition functions, we get principal $K$-bundle over $M$. The pullbacks of the Maurer Cartan form restricted to those pieces can be glued together to a Cartan connection on this $K$-bundle. Thus we get a locally flat Cartan geometry of type $(G,K)$.

Now $G/K$ is a connected, simply connected and $S_x$ acts as an automorphism on $G$. If it is not an inner automorphism, we can extend $G$ and $K$ by $h:=S_x$ and $(G,K)$ still satisfies (H3). Clearly $h$ satisfies (H1) and (H2). It is obvious that the local reflexions are equivalent to those defined in the first example.
\end{proof}

\end{document}